\documentclass{article}
\usepackage{amsthm,amssymb,amsmath}
\usepackage{graphicx}
\usepackage{float}
\usepackage{tikz}
\usepackage{mathpazo}

\theoremstyle{plain}
\newtheorem{thm}{Theorem}
\newtheorem{cor}[thm]{Corollary}
\newtheorem{lem}[thm]{Lemma}
\newtheorem{prop}[thm]{Proposition}

\theoremstyle{definition}
\newtheorem{defn}[thm]{Definition}

\newcommand{\Z}{\mathbb{Z}^+}
\newcommand{\Ex}{\mathbb{E}(\chi)}
\newcommand{\Em}{\mathbb{E}(\varphi)}
\newcommand{\dsum}{\sum_{n\geq 0} \sum_{k\geq 0}}
\newcommand{\Exs}{\mathbb{E}(\chi ^2)}
\newcommand{\Vx}{\mathbb{V}(\chi)}
\newcommand{\Vm}{\mathbb{V}(\varphi)}

\newcommand{\s}{\hspace{.5pc}}
\def\ds{\displaystyle}

\begin{document}

\title{Counting Fixed-Length Permutation Patterns}
\author{Cheyne Homberger \\
\small Department of Mathematics \\[-0.8ex]
\small University of Florida \\[-0.8ex]
\small Gainesville, FL \\
\small\tt cheyne42@ufl.edu }

\maketitle

\begin{abstract} 
  We consider the problem of packing fixed-length patterns into a permutation,
  and develop a connection between the number of large patterns and the number
  of bonds in a permutation. Improving upon a result of Kaplansky and Wolfowitz,
  we obtain exact values for the expectation and variance for the number of
  large patterns in a random permutation. Finally, we are able to generalize the
  idea of bonds to obtain results on fixed-length patterns of any size, and
  present a construction that maximizes the number of distinct large patterns.  
\end{abstract}

\section{Background}
  
  Two sequences of distinct integers $a_1 a_2 \ldots a_n$ and $b_1 b_2 \ldots
  b_n$ are \emph{order isomorphic} if, for all $1 \leq i,j \leq $n, we have that
  $a_i < a_j$ if and only if $b_i < b_j$.  Let $q = q_1 q_2 \ldots q_k$ be a
  permutation in the symmetric group $S_k$ written in one-line notation. We say
  that a permutation $p = p_1 p_2 \ldots p_n \in S_n$ \emph{contains $q$ as a
  pattern} if there is a subsequence $p_{i_1} p_{i_2} \ldots p_{i_k}$ which is
  in the same relative order as the entries of $q$. If $p$ does not contain $q$
  as a pattern, we say that $p$ \emph{avoids} $q$. 

  For example, the permutation $p = 4732615$ contains the pattern $q = 213$ because
  the 1st, 4th, and 7th entries of $p$ are order isomorphic to the permutation
  $213$. This permutation avoids the pattern $123$, however, because $p$
  contains no increasing subsequence of length $3$. For another example, a
  permutation $p$ avoids the pattern $q = 21$ if and only if it is strictly
  increasing, since otherwise $p$ would contain an inversion, and an inversion
  is precisely a $21$ pattern.
  
  As a relation, pattern containment is transitive, reflexive, and
  anti-symmetric. Therefore the set of all permutations equipped with this ordering
  forms a graded partially ordered set (poset), which is referred to
  in the literature as the \emph{pattern poset}. Given a permutation $p$, the
  set of all patterns contained in $p$ forms a downset (also referred to as an
  ideal) of this poset.

  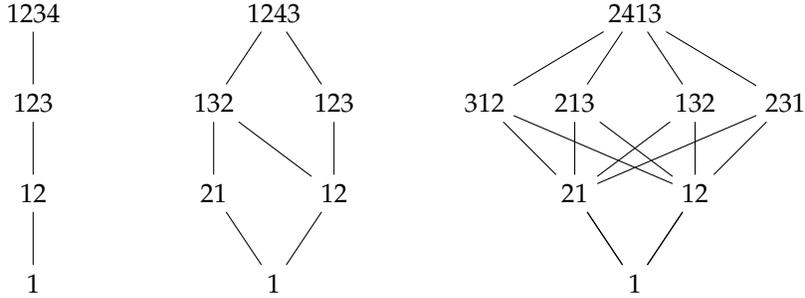
\begin{figure}[ht]
  \centering
  \begin{tikzpicture}
  [scale = .4]
    \node (a1) at (-8,9) {$1234$};
    \node (a2) at (-8,6) {$123$};
    \node (a3) at (-8,3) {$12$};
    \node (a4) at (-8,0) {$1$};
    \draw (a1) -- (a2) -- (a3) -- (a4);

    \node (b1) at (0,9) {$1243$};
    \node (b21) at (-2,6) {$132$};
    \node (b22) at (2,6) {$123$};
    \node (b31) at (-2,3) {$21$};
    \node (b32) at (2,3) {$12$};
    \node (b4) at (0,0) {$1$};
    \draw (b1) -- (b21) -- (b31) -- (b4);
    \draw (b1) -- (b22) -- (b32) -- (b4);
    \draw (b21) -- (b32);

    \node (c1) at (12,9) {$2413$};
    \node (c21) at (7,6) {$312$};
    \node (c22) at (10,6) {$213$};
    \node (c23) at (14,6) {$132$};
    \node (c24) at (17,6) {$231$};
    \node (c31) at (10,3) {$21$};
    \node (c32) at (14,3) {$12$};
    \node (c4) at (12,0) {$1$};
    \draw (c1) -- (c21) -- (c31) -- (c4);
    \draw (c1) -- (c22) -- (c31) -- (c4);
    \draw (c1) -- (c23) -- (c32) -- (c4);
    \draw (c1) -- (c24) -- (c32) -- (c4);
    \draw (c21) -- (c32) --(c22);
    \draw (c23) -- (c31) --(c24);
  \end{tikzpicture}
  \caption{Downsets of 1234, 1243, and 2413}
  \end{figure}

  
  The area of permutation patterns has received considerable attention in recent
  years. The majority of work has been focused on enumerating infinite downsets in
  the pattern poset (known as permutation classes), particularly those which
  arise as sets of permutations avoiding specified patterns. 
  An early result in the area, due to Knuth \cite{knuth3}, is that the 231 avoiding
  permutations are counted by the Catalan numbers $\frac{1}{n}\binom{2n}{n}$,
  and these are exactly the stack sortable permutations.  A more comprehensive
  introduction to the subject can be found in \cite{bonabook}. 
  
  Interesting questions are raised, however, even if we restrict ourselves to
  finite downsets of the pattern poset.  We focus our attention here on
  examining the downset of a single permutation.  
  In 2003, Herb Wilf raised the question of finding the maximum number of distinct
  patterns which can be contained in a permutation of length $n$, and
  classifying those permutations which achieve this maximum.
  Translated to the language of posets, Wilf's question asks to find which
  permutations maximize the size of their downset in the pattern poset. 
  In \cite{albert}, the authors showed that the maximum number of patterns that
  can be contained in a permutation of length $n$ is asymptotic to $2^n$.
  However, the exact value of the maximum is unknown. 

\section{Preliminaries}
  
  This paper can be divided into
  two parts: in the first, we examine the number of $(n-1)$-patterns in a random
  $n$-permutation, and obtain exact values for both the expectation and variance
  of this statistic by extending a 1945 result of Kaplansky and Wolfowitz. In
  the second part, we examine the number of patterns of a fixed size in a given
  permutation, and provide a partial answer Herb Wilf's question.

  In counting the total number of patterns contained in a permutation, it is
  most useful to use a top-down approach, enumerating all of the largest
  patterns and working down the downset level by level. We introduce an
  alternate (but equivalent) definition of permutation patterns which better
  suits this approach.

  \begin{defn} 
  Let $p = p_1p_2 \ldots p_n \in S_n$, $n\geq 2$. Let $0 \leq k \leq
  n-1$. We say that a permutation $q \in S_{n-k}$ is an \emph{$(n-k)$-pattern} of
  $p$ if $q$ can be obtained by deleting $k$ entries of $p$ and then relabelling
  the remaining entries $1$ through $n-k$ with respect to order. Let
  \emph{$D_k(p)$} denote the set of $(n-k)$-patterns of a permutation $p$, and
  $\mathcal{D}(p) = \bigcup_k D_k(p)$ denote the set of \emph{all} patterns
  contained in $p$.
  \end{defn}

   
 Where Wilf's problem asks
  which permutations maximize the total size of the downset, our focus will be
  on finding those permutations which maximize the width of a \emph{specified}
  level of this downset. 
  
  First, we investigate the number of coatoms. That is, we will fix an $n\geq 2$
  and focus our attention on patterns of size $(n-1)$ contained in a given
  $n$-permutation $p$. We limit our attention to $(n-1)$-patterns not only
  because they are easier to work with, but because these results can in some
  cases be extended to results for $(n-k)$-patterns, simply by working our way
  down the downset level by level. To start, we formalize our notion of
  $(n-1)$-patterns with a function.  To simplify the notation, we use $[n]$ to
  denote the set $\{1,2,3, \ldots n\}$. 

  \begin{defn} 
  Let $del: S_n\times [n] \rightarrow S_{n-1}$ be the function where
  $del(p,i)$ is defined by deleting the $i$th entry of p, and relabelling the
  remaining entries $1$ through $n-1$ with respect to order.  
  \end{defn} 

  Since any $(n-1)$-pattern $q$ of a permutation $p$ uses all but one entry of
  $p$, we see that $q = del(p,k)$ for some $k \in [n]$. Also, it is clear that if $q
  = del(p,k)$ for some $k \in [n]$, then $q$ is contained in $p$ as a pattern. This
  implies that $D_1(p) = \{del(p,k) : k\in [n] \}$.

  Inversely, we can build up an $(n-1)$-permutation into an $n$-permutation by
  inserting an extra entry.
  We define another function to formalize this idea.

  \begin{defn} 
  Let $ins : S_{n-1}\times [n] \times [n] \rightarrow S_n$ be the function where
  $ins(q,j,k)$ is defined by inserting the entry $k-1/2$ immediately to the left
  of the $j$th entry of $q$, and then relabelling $1$ through $n$ with respect
  to order. Let $I_1(q) = \bigcup_i \bigcup_j ins(q,i,j)$ denote the set of all
  $n$-permutations which can be obtained by inserting one entry to $q$. 
  \end{defn}

  The function $ins$ can best be understood graphically:

    \begin{figure}[ht]
      \centering
      \begin{tikzpicture}
      [scale = .8, line width=1pt]
      \draw (0,5.5) -- (0,0) -- (5.5,0);
      \foreach \x in {1,2,3,4,5}
        \draw (\x, 0) -- (\x, -.2);
      \foreach \y in {1,2,3,4,5}
        \draw (0, \y) -- (-.2, \y);
      \foreach \x/\y in {1/1, 2/5, 3/3, 4/2, 5/4}
        \draw[fill = black] (\x,\y) circle (1.7mm);
      \end{tikzpicture} \hspace{4pc}
      \begin{tikzpicture}
      [scale = .8, line width = 1pt]
      \draw (0,5.5) -- (0,0) -- (5.5,0);
      \foreach \x in {1,2,3,4,5}
        \draw (\x, 0) -- (\x, -.2);
      \foreach \y in {1,2,3,4,5}
        \draw (0, \y) -- (-.2, \y);
      \foreach \x/\y in {1/1, 2/5, 3/3, 4/2, 5/4, 1.5/3.5}
        \draw[fill = black] (\x,\y) circle (1.7mm);
      \draw[line width = .5pt] (1.5,0) -- (1.5, 5.5);
      \draw[line width = .5pt] (0, 3.5) -- (5.5, 3.5);
      \end{tikzpicture}
   \caption{$ins(15324,2,4) = 146325$.}
   \end{figure}
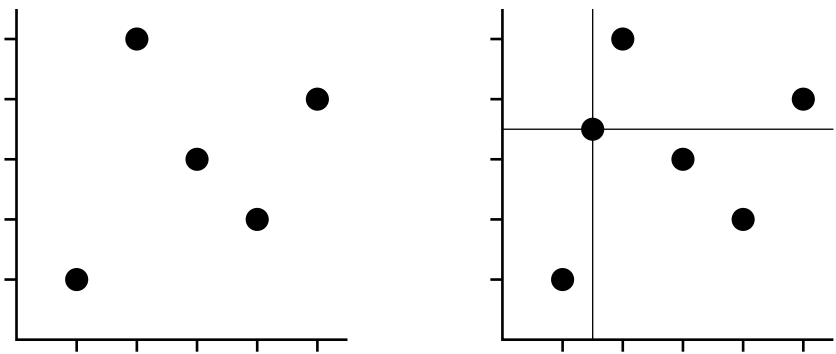

  Now, from the definitions of these two functions, we see that they satisfy the
  following inverse relationship: 
  $$del(ins(q,j,k),j) = q \text{ and } ins(del(p,i),i,p_i) = p.$$

  This relationship, along with the fact that $D_1(p)$ is the set of all
  $(n-1)$-patterns contained in $p$, implies that $I_1(q)$ is exactly the set of all
  $n$-permutations which contain $q$ as a pattern.

\section{The size of $D_1(p)$}

  It follows directly from the definition that given any permutation $ p \in S_n$,
  $|D_1(p)| \leq n$, and that $|D_1(p)| = n$ if and only if $del(p,i) = del(p,j)$
  implies that $i = j$. Before investigating further, we introduce another
  definition.

  \begin{defn} 
  Let $p = p_1p_2 \ldots p_n$ be a permutation, and let $i \in [n-1]$. Say that
  the pair $(p_i, p_{i+1})$ is a \emph{bond}, of entries of $p$ if $p_i - p_{i+1}
  = \pm1$. We say that the sequence $(p_i, p_{i+1}, \ldots p_{i+k-1})$ is a
  \emph{run} of length $k$ if, for $1 \leq j \leq k-2$, the pair
  $(p_{i+j},p_{i+j+1})$ is a bond. Denote by $C(p)$ the number of bonds contained
  in a permutation $p$. 
  \end{defn}

  Note that runs are necessarily either increasing or decreasing, and that a run
  of length $k$ contains $k-1$ bonds. We can now establish a fundamental
  relationship between bonds and $(n-1)$-patterns. 

  \begin{lem} \label{biglem} 
  Let $p = p_1p_2 \ldots p_n$, and $1 \leq j < k \leq
  n$. Then $del(p,j) = del(p,k)$ if and only if $p_j$ and $p_k$ are part of the same
  run.  
  \end{lem} 
  \begin{proof} The forward direction is clear, since removing any element of a
  run and relabelling simply results in a shorter run. 

  The other direction takes a bit more work. Suppose that there are $j,k$ with $1
  \leq j < k \leq n$ with $del(p,j) = del(p,k)$. We proceed by induction on $k-j$. 

  Suppose that $k =j+ 1$. Assume first that $p_j < p_{j+1}$, and consider the
  $j$th entry of $del(p,j)=del(p,j+1)$. By the definition of $del$, the $j$th
  entry of $del(p,j)$ is $p_{j+1}-1$, and the same entry in $del(p,j+1)$ is
  $p_j$. Therefore, we see that $p_{j+1}-1=p_j$, which means that
  $(p_j,p_{j+1})$ is a bond. Again, the case where $p_{j+1}<p_j$ follows
  similarly. 

  Now assume by way of induction that the statement holds when $k=j+m-1$, and
  suppose there exists $1\leq j < k \leq n$ such that $k-j = m$ and
  $del(p,j)=del(p,k)$. Assume first that $p_j < p_{k}$. $del(p,j)=del(p,k)$ implies, in
  particular, that the $(k-1)$st entries on both sides of the equality are equal.
  By definition, the $k-1$ entry of $del(p,j)$ is $p_k-1$, while the $k-1$ entry of
  $del(p,k)$ is either $p_{k-1}$ or $p_{k-1}-1$. The latter case would imply that
  $p_{k-1} = p_k$, a contradiction, and so it follows that $p_{k-1} = p_k$. 

  By what has already been proved, $del(p,k-1) = del(p,k)$ since these entries form a
  bond. But then $del(p,j) = del(p,k) = del(p,k-1)$, and so by the induction hypothesis
  the entries $(p_j p_{j+1} \ldots p_{k-1})$ form a run. Finally, $p_k - 1 =
  p_{k-1}$ implies that $(p_j p_{j+1} \ldots p_{k-1} p_k)$ is a length $m$ run.
  Once more, the case where $p_j > p_k$ follows similarly, and the lemma is
  proved.

  \end{proof}

  The simplest examples of permutations with runs are the ascending and descending
  permutations. Removing any element from the ascending (descending) permutation
  of length $n$ and renumbering results in the ascending (descending) permutation
  of length $n-1$. In other words, the $(n-1)$-pattern set of either of these
  permutations has size $1$, and the lemma shows that these are the only
  permutations with this property. 

  We can now establish our connection between the number of bonds and the number
  of $(n-1)$-patterns. Lemma \ref{biglem} directly implies the following theorem.

  \begin{thm} \label{bigthm} Let $p \in S_n$. Then $|D_1(p)| = n- C(p)$. 
  \end{thm}

  This leads to a number of useful corollaries. The first is clear, and provides
  motivation for generalization.

  \begin{cor} \label{distinct} 
  A permutation has the maximum number $(n-1)$-patterns if
  and only if it contains no bonds. 
  \end{cor}

  Theorem \ref{bigthm} also provides a simple proof of the following local
  property of the permutation pattern poset.

  \begin{cor} \label{n2+1} 
  If $q \in S_{n-1}$, then $|I_1(q)| = n^2 - 2n + 2 $. In other words, every
  $(n-1)$-permutation is contained in exactly $(n-1)^2 +1$ $n$-permutations. 
  \begin{proof}

  By definition, the set $I_1(q) = \{ins(q,j,k) : 1 \leq j , k \leq n\}$, so we see
  that $|I_1(q)| \leq n^2$.

  Now, a permutation $p \in S_n$ is contained in $I_1(q)$ more than once exactly
  when $q$ can be obtained in more than one way by deleting a entry of $p$. It
  follows that $q$ is contained in a permutation $p \in S_n$ more than once
  exactly when $ins(q,j,k) = ins(q,j',k')$ where $(j,k) \not = (j',k')$. By the lemma,
  this happens exactly when the $j$th entry of $ins(q,j,k)$ is a part of the same
  run as the $j'$ entry of $(ins(q,j',k'))$. We can prevent this from occuring by
  never inserting an element just to the right and directly above or below an
  existing element of $q$, as this ensures that any new bonds can be created in
  exactly one way. 

  This eliminates exactly $2(n-1)$ choices for inserting an entry into $q$, and so
  therefore $|I_1(q)| = n^2 - 2(n-1) = (n-1)^2 +1$, and the proof
  is complete.  
  \end{proof}
  \end{cor}

\section{Expectation and Variance of $|D_1(p)|$}

  We now examine the distribution of the number of $(n-1)$-patterns in a
  randomly chosen $n$-permutation $p$ by first examining the distribution of
  bonds. Kaplansky and Wolfowitz presented in \cite{kap} and \cite{wolf} the
  asymptotic distribution of the number of bonds in a random permutation. Using
  more modern techniques of generating function analysis we are able to improve
  upon their results and obtain \emph{exact} formulas for the expectation and
  the variance for the number of bonds in a random permutation. Theorem
  \ref{bigthm} allows us to translate these into the results on fixed-length
  patterns.

  Throughout this section, we will let $\varphi:S_n \rightarrow \mathbb{Z}^{\geq
  0}$ be the variable indicating the number of distinct $(n-1)$-patterns of an
  $n$-permutation, and $\chi:S_n \rightarrow \mathbb{Z}^{\geq 0}$ be the variable
  indicating the number of bonds. Our main tool will be multivariate generating
  functions, but first we note that the $\Em$ can be obtained directly using our
  connection to pattern containment.

  \begin{prop} \label{expect}
  The expectation $\Em =n -  \frac{2(n-1)}{n}$, which approaches $n-2$ as $n$
  increases.  
  \end{prop}
  \begin{proof}
  This follows immediately from Corollary \ref{n2+1} and the identity
  $$(n-1)!(n^2-2n+2) = n!\left(n-\frac{2(n-1)}{n}\right).\qedhere $$ 
  \end{proof}

  Generating functions, however, allow us to go several steps further. It follows
  from Theorem \ref{bigthm} and the linearity of expectation that $\Em = n - \Ex$,
  which allows us to easily translate results about bonds into results about
  distinct $(n-1)$-patterns. We can now begin the construction of our multivariate
  generating function, using a technique similar to the cluster method of Goulden and
  Jackson.

  \begin{thm} \label{genfcn} 
  Let $a_{n,k}$ be the number of permutations of length $n$ which contain exactly
  $k$ bonds, and set $a_{0,0} = 1$. Then we have that $$F(z,u) := \dsum a_{n,k}
  x^n u^k = \sum_{m\geq 0} m!\left(z+\frac{2z^2(u-1)}{1-z(u-1)}\right)^m.$$
  \end{thm}

  \begin{proof}
  First, we construct a generating function $G(z,u)= \dsum b_{n,k} x^n u^k$, where
  $b_{n,k}$ is the number of permutations of length $n$ with $k$
  \emph{distinguished} bonds. For example, $b_{n,0} = n!$, as every permutation
  can be written with no bonds distinguished, and no permutation is counted more
  than once. 

  The function $G(z,u)$ is easier to construct, as we can build an $n$-permutation
  with $k$ distinguished bonds by first specifying our distinguished ascending and
  descending runs, then permuting these runs with the remaining entries. Now, a
  run of length $j$ contains $j-1$ bonds, and we have the option of making each
  run either increasing or decreasing. This leads to 

  $$G(z,u) = \sum_{m\geq 0} m! \left(z+ \frac{2z^2 u}{1-zu} \right)^m.$$

  Now we can use the function $G$ to obtain a formula for $F$. Since $G$ counts
  only the distinguished bonds and $F$ counts every bond, we see that $F$ and $G$
  are related by the transformation $F(z,u+1) = G(z,u)$. Therefore $F(z,u) =
  G(z,u-1)$, and the theorem is proved.  
  \end{proof} 

  A simple transformation can be used to obtain a multivariate generating function
  which indicates the number of distinct $(n-1)$-patterns of a permutation.
  However, the function $F$ is more useful to work with, as we will see soon. 

  \begin{cor} 
    Let $d_{n,k}$ be the number of permutations of length $n$ with exactly $k$
    distinct $(n-1)$-patterns. Then $$ H(z,u) = 1+  \sum_{n\geq k \geq 1} d_{n,k}
    x^n u^k = \sum_{m\geq
    0}m!\left(zu+\frac{2zu^2(u^{-1}-1)}{1-zu(u^{-1}-1)}\right)^m. $$
  \end{cor}
  \begin{proof}
    Since $|M(p)| = n-C(p)$, it follows immediately that $H(z,u) = F(zu, u^{-1})$. 
  \end{proof}

  We can now coax several results out of the function $F(z,u)$. To start,
  plugging in $u=0$ gives the generating function for permutations with no bonds.
  Expanding, we see that 
  $$F(z,0) = 1+z+2z^4+14z^5+90z^6+646z^7 + 5242z^8 \ldots.$$

  The sequence, $1,1,0,0,2,14,90,646,5242, \ldots$ is A002464 in the OEIS, and is
  easily seen to be equal to the number of ways to place $n$ non-attacking kings
  on an $n \times n$ chessboard with one king in each row and each column. It was
  shown in \cite{tauraso} that this sequence is asymptotic to $n!/e^2$, and so
  Corollary \ref{distinct} implies the following. 

  \begin{prop} 
  The probability that a randomly selected $n$-permutation has all distinct
  $(n-1)$-patterns tends to $1/e^2$ as $n \rightarrow \infty$.  
  \end{prop}

  We can take this a step further, and use the function $F(z,u)$ to determine the
  expected number of bonds in a random $n$-permutation. As described in
  \cite{flajolet}, we have that 

  $$ \Ex = \frac{[z^n] \partial_u F(z,u) |_{u=1}}{n!} = \frac{1}{n!}[z^n]
  \partial_u \left( \sum_{m\geq 0} m!\left(z+\frac{2z^2(u-1)}{1-z(u-1)}\right)^m
  \right). $$  
   
  Taking the partial derivative with respect to $u$, we find that 

  $$\partial_u F(z,u) = \partial_u \left( \sum_{m\geq 0} m!\left(
  z+\frac{2z^2(u-1)}{1-z(u-1)} \right)^m \right) = $$

  $$
  \sum_{m\geq 0} m \cdot m! \frac{ \left(\ds z+ \frac{2z^2(u-1)}{1-z(u-1)}
  \right)^m \left( \ds \frac{2z^2}{1-z(u-1)} + \frac{2z^3(u-1)}{(1-z(u-1))^2}
  \right) }{\ds z+ \frac{2z^2(u-1)}{1-z(u-1)}}.$$

  Plugging in $u=1$ simplifies this expression greatly, leaving

  $$\partial_u F(z,u)|_{u=1} =  \sum_{m \geq 0} 2m! \cdot m z^{m+1} = 
  \sum_{m\geq 1} 2(m-1)! \cdot (m-1) z^m .$$

  From this it follows that 

  $$\Ex = \frac{[z^n] \partial_u F(z,u) |_{u=1}}{n!} 
  = 2\frac{(n-1)! \cdot (n-1)}{n!} 
  = 2\frac{(n-1)}{n}.$$

  Finally, by using linearity of expectation and the fact that $\varphi = n- \chi$,
  we find that $\Em = n - \Ex = n - 2\frac{n-1}{n}$, in agreement with Proposition
  \ref{expect}.

  The variance is given by $\Vx = \Exs - \Ex ^2$, and so we find that 
  $$ \Vx = \mathbb{E}(\chi(\chi - 1)) + \mathbb{E}(\chi) - \mathbb(\chi)^2.$$ 
  The factorial moment can be computed directly from the bivariate generating
  function $F$ as follows:
  $$\mathbb{E}(\chi (\chi - 1)) = \frac{[z^n] \partial_u ^2 F(z,u)
  |_{u=1}}{n!}.$$ 

  This leads to 

  $$ \begin{aligned} \Vx 
    &= \frac{[z^n] \partial_u ^2 F(z,u)
  |_{u=1}}{n!} +\frac{[z^n] \partial_u F(z,u) |_{u=1}}{n!} - \left(\frac{[z^n]
  \partial_u  F(z,u) |_{u=1}}{n!} \right) ^2 \\ 
    &= \frac{[z^n] \partial_u ^2
  F(z,u) |_{u=1}}{n!} + 2\frac{n-1}{n} - \left( 2 \frac{n-1}{n} \right) ^2.
  \end{aligned} $$

  We begin by taking the second derivative of $F(z,u)$ with respect to $u$, which
  gives:

  $$\begin{aligned} 
  \partial_u F(z,u) = \ds \sum_{m\geq 0}m! \cdot m & 
  \left( \frac{m \left(z+ \frac{2z^2(u-1)}{1-z(u-1)} \right)^m
  \left(\frac{2z^2}{1-z(u-1)} + \frac{2z^3(u-1)}{(1-z(u-1))^2} \right)^2}{
  \left(z+ \frac{2z^2(u-1)}{1-z(u-1)} \right)^2} \right. \\ 
    &  + \left.
    \frac{\left(z+ \frac{2z^2(u-1)}{1-z(u-1)}\right)^m
  \left(\frac{4z^3}{(1-z(u-1))^2} + \frac{4z^4 (u-1)}{(1-z(u-1))^3} \right)}{
  z+\frac{2z^2(u-1)}{1-z(u-1)}} \right. \\ 
  & - \left.
  \frac{\left(z+\frac{2z^2(u-1)}{1-z(u-1)} \right)^m \left(\frac{2z^2}{1-z(u-1)} +
  \frac{2z^3(u-1)}{(1-z(u-1))^2} \right)^2}{ \left(z+\frac{2z^2(u-1)}{1-z(u-1)}
  \right)^2} \right) . 
   \end{aligned} $$

  Once again, setting $u=1$ simplifies this expression immensely:

  $$ \partial_u ^2 F(z,u) |_{u=1} = \ds \sum_{m\geq 0} 4m! \cdot m^2 z^{m+2} \\
  = \sum_{m \geq 2} 4(m-2)!(m-2)^2z^m. $$

  Which produces:

  $$\begin{aligned} \Vx &= \frac{[z^n] \partial_u ^2 F(z,u) |_{u=1}}{n!} +
  2\frac{n-1}{n} - \left( 2 \frac{n-1}{n} \right) ^2\\ 
  &= \frac{4(n-2)!(n-2)^2}{n!} + 2\frac{n-1}{n} - \left( 2 \frac{n-1}{n} \right)^2 \\
  &= 4\frac{(n-2)^2}{n(n-1)}+  2\frac{n-1}{n} -  4 \frac{(n-1)^2}{n^2}.
  \end{aligned}$$

  Which converges to 2 for large $n$. From the fact that $\varphi = n- \chi$, it
  follows that $\Vm = \Vx$. We summarize this in the following theorem.

  \begin{thm} 
  Let $m: S_n \rightarrow \Z$ be the variable indicating the number of distinct
  $(n-1)$-patterns of a permutation $p \in S_n$. Then we have: 

  $$\Em = n - 2\frac{n-1}{n} \text{ and }
  \Vm = 4\frac{(n-2)^2}{n(n-1)}+  2\frac{n-1}{n} -  4 \frac{(n-1)^2}{n^2}.$$ 

  \end{thm}

  An immediate consequence, we see that for large $n$ these approach $\Em = n-2$
  and $\Vm = 2$ respectively, implying as a special case the results of \cite{kap}
  and \cite{wolf}. These same techniques can be applied to recursively calculate
  higher moments.

\section{Patterns of other sizes}
	
  We turn our attention now to determining the number $|D_k(p)|$ of distinct
  $(n-k)$-patterns of a permutation, for $k >1$. In particular, we seek to
  determine which permutations (if any) have the property that $|D_k(p)| =
  \binom{n}{k}$, the maximum number of possible $(n-k)$-patterns. To start, we
  generalize our notion of bonds with the following definition. 
                  
  \begin{defn} 
  Let $p = p_1p_2 \ldots p_n \in S_n$ be any permutation. Define a metric on the
  entries of $p$ by $d_p(i,j) = |i-j|+|p_i - p_j|$. Define the \emph{minimum gap}
  of a permutation $p$ to be $mg(p) = min\{d_p(i,j):1\leq i < j  \leq n\}$.
  \end{defn}

  If the permutation is plotted on a lattice, then the metric $d$ is just the
  taxicab metric on $\mathbb{Z}^2$. It is easy to see that $(p_i,p_j)$ is a bond
  if and only if $d(i,j) = 2$. Therefore, we see that $p$ has all distinct
  $(n-1)$-patterns if and only if $mg(p) \geq 3$. This motivates a
  generalization of Corollary \ref{distinct}, after we establish some suitable
  notation.

  \begin{defn} 
  Let $S=\{a_1,a_2, \ldots a_k\}\subseteq [n]$, with $1 \leq a_1<a_2 < \ldots <a_k
  \leq n$. We denote $del( \ldots del(del(del(p,a_k),a_{k-1}),a_{k-2}), \ldots
  ,a_1)$ by $del(p;S)$.  In other words, to obtain $del(p;S)$ we remove
  $p_{a_1}, p_{a_2}, \ldots p_{a_k}$ from $p$ and renumber the remaining entries
  with respect to order.
  \end{defn}

  \begin{defn} 
  Let $p=p_1p_2 \ldots p_n$ be a permutation, and $1\leq i < j \leq n$ then the
  \emph{span} of the entries $p_i$ and $p_j$ is denoted $span_p(i,j)$ and is
  defined to be the set of indices for the entries that lie between $p_i$
  and $p_j$ either vertically or horizontally. 

  Formally, if $p_i < p_j$, then $span_p(i,j) = \{k : i<k<j \text{ \ or \ } p_i
  < p_k < p_j \}$, with a similar definition when $p_i> p_j$.  
  \end{defn}

  \begin{lem} \label{span}
  If $p=p_1p_2 \ldots p_n$ is a permutation with $mg(p) = k$, and if $1\leq i
  < j \leq n$ are such that $d_p(i,j) = k$, then $|span_p(i,j)| = k-2$.
  \end{lem}
  \begin{proof}
  It is clear that $|span_p(i,j)| \leq k-2$. The only way in which $|span_p(i,j)|
  < k-2$ would hold is if there existed an entry $p_m$ which was in between $p_i$
  and $p_j$ both vertically and horizontally. However, this $p_m$ would contradict
  the minimality of $k$, so $|span_p(i,j)| = k-2$.  
  \end{proof}

  Corollary \ref{ratchet} now follows immediately from the lemma.

  \begin{cor} \label{ratchet} 
  If $p=p_1p_2 \ldots p_n$ is an $n$-permutation with $mg(p) =k$, then
  $mg(del(p,i)) \geq k-1$ for all $i \in [n]$.
  \end{cor}

  We are now able to prove our generalization of Corollary \ref{distinct}.

  \begin{thm} \label{otherthm} 
  Let $p \in S_n$. Then $p$ has all distinct $(n-k)$-patterns if and only if
  $mg(p) \geq k+2$.  
  \begin{proof}

  We start with the forward direction. Let $p=p_1p_2 \ldots p_n$ be a
  permutation with the maximum number of $(n-k)$-patterns. Assume, by way of
  contradiction that $mg(p) = m < k+2$. Let $i < j$ be such that $d_p(i,j) = m$.
  By Lemma \ref{span}, we have that $span_p(i,j) = \{a_1, a_2, \ldots a_{m-2}
  \}$. Now set $q = del(p;\{ a_1,a_2, \ldots a_{m-2}\})\in S_{n-m+2}$. It
  follows $mg(q)=2$, and so $q$ has a consecutive pair, which implies that $q$
  does not the maximum number $(n-1)$-patterns. Since $m-2 < k$, $p$ cannot have
  the maximum number of $(n-k)$-patterns, a contradiction.

  For the reverse implication, let $p=p_1p_2 \ldots p_n$ be a permutation with
  $mg(p) \geq k+2$. We use induction on $k$. We have already seen that the
  statement is true for $k=1$, so assume the statement holds for all positive
  integer less than $k$. Let $p=p_1p_2 \ldots p_n$ be a permutation with $mg(p)
  = k+2$. By induction, we know that this permutation has all distinct
  $(n-m)$-patterns for all $1\leq m <k$. 
    
  Suppose, by way of contradiction, that $q \in S_{n-k}$ is contained in $p$ in
  two different ways. That is, suppose that $del(p;\{a_1,a_2, \ldots a_k\}) = q
  = del(p;\{b_1,b_2, \ldots b_k\})$, with $a_i < a_j$ and $b_i<b_j$ when $i<j$,
  and $A = \{a_1, a_2, \ldots a_k \} \not = \{b_1, b_2, \ldots b_k\} = B$. 
    
  Now we claim that $A \cap B = \emptyset$. To see this, we can suppose that
  $a_i = b_j$. But then $q$ is contained as a pattern in two different ways in
  $del(p,a_i)$. However this contradicts $mg(del(p,a_i)) \geq k+1$, because by
  induction it has the maximum number of $((n-1)-(k-1)) = (n-k)$-patterns. 
    
  Assume, without loss of generality, that
  $a_1 < b_1$. Let $j \in [n]$ be the smallest value such that $j>a_1$ but $j
  \notin A$. Since $del(p;\{a_1,a_2, \ldots a_k\}) = del(p;\{b_1,b_2, \ldots
  b_k\}) = q = q_1q_2 \ldots q_{n-k}$, it follows that $p_{a_1}$ and $p_{j}$
  will both move to fulfill the role of $q_{a_1}$ once the entries from $A$ or
  $B$ are removed and the permutation is renumbered. However, since $|A| = k$,
  this implies that $d(a_1, j) < k+2$, our final contradiction. 
    
  \end{proof} \end{thm}

  \begin{cor} \label{easy} 
  Let $p\in S_n$. If $|D_k(p)|= \binom{n}{k}$, then $D_j(p) = \binom{n}{j}$ for
  all $j \in [k]$.  
  \end{cor}

  To see that permutations with arbitrarily large gap sizes exist, we first note
  that the slanted cube construction presented in \cite{albert} creates a
  permutation of length $n^2$ minimum gap equal to $n+1$. We will construct a
  sequence of permutations $\{\pi^{(n)}\}_{n=2}^\infty$ which does slightly better,
  creating the same minimum gap size with shorter length. 

  To build the permutation $\pi^{(n)}$ with gap size $n$, we begin with a tiling of
  the plane with squares with side length $n$. Then we simply rotate the tiling by
  $45$ degrees and use the centers of the squares as our permutation entries. We
  define this formally as follows.

  \begin{defn} 
  Let $a_i$ be defined as 
  $$a_i = min\{d \in [k-1] : i \leq d(k-1) \} \text{ \  and  \ } b_i = (i-1 \s
  \text{ mod } (k-1)) \cdot (k-1).$$ 
  Define $p_i = a_i+b_i$. Now take the
  permutation $p'=p_1 p_2 \ldots \pi^{(k-1)^2}$, and define $\pi^{(n)} =
  del(p',1,(k-1)^2)$, the permutation obtained by deleting the first and last
  entries of $p'$.
  \end{defn}

  The permutation $\pi^{(k)}$ for $k=4,5$ is shown below, and it is clear that these
  permutations have minimum gap size equal to $4$ and $5$ respectively. It is
  clear that $\pi^{(k)}$ is an involution for all $k$, and that the complement of
  $\pi^{(k)}$ is equal to it's own reverse. Therefore, the orbit of $\pi^{(k)}$ under
  the automorphism group of the pattern poset has order 2.

  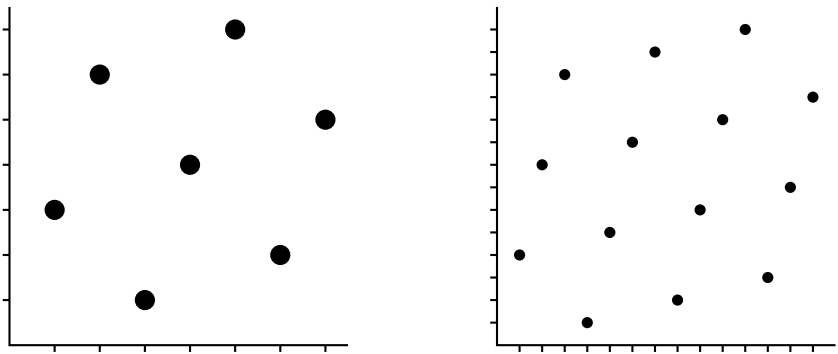
\begin{figure}[ht] \label{slant}
  \centering
  \begin{tikzpicture}
    [scale = .3, line width = .8pt]
    \draw (0,15) -- (0,0) -- (15,0);
    \foreach \num in {2, 4,..., 14}
    {
      \draw (\num, 0) -- (\num, -.3);
      \draw (0, \num) -- (-.3, \num);
    }
    \foreach \y [count = \x] in {3,6,1,4,7,2,5}
      \draw[fill = black] (2*\x,2*\y) circle (4mm);
  \end{tikzpicture} \hspace{4pc}
  \begin{tikzpicture}
    [scale = .3, line width = .8pt]
    \draw (0,15) -- (0,0) -- (15,0);
    \foreach \num in {1,..., 14}
    {
      \draw (\num, 0) -- (\num, -.3);
      \draw (0, \num) -- (-.3, \num);
    }
    \foreach \y [count = \x] in {4,8,12,1,5,9,13,2,6,10,14,3,7,11}
      \draw[fill = black] (\x,\y) circle (2mm);
  \end{tikzpicture}
   \caption{$\pi^{(4)} = 3 \ 6 \ 1 \ 4 \ 7 \ 2 \ 5$, and $\pi^{(5)} = 4 \ 8 \ 12 \ 1 \
   5 \ 9 \ 13 \ 2 \ 6 \ 10 \ 14 \ 3 \ 7 \ 11$} \label{pk}
   \end{figure}

  A permutation can be embedded in the plane and the metric $d_p$ can be extended
  to the taxicab metric $d_1$ on $\mathbb{R}^2$. It follows that any permutation
  $p$ with minimum gap size equal to $k$ defines a tiling of the plane by tilted
  squares with side lengths equal to $k$ and centers on points of $\mathbb{Z}^2$.
  It is clear that a minimal sized permutation with gap size equal to $k$ will
  produce a maximum tiling of the plane with tilted squares centered on different
  horizontal and vertical lines. There are exactly two such tilings of the plane,
  corresponding to the permutations $\pi^{(k)}$ and its reverse. We summarize this
  in the following theorem.

  \begin{thm}
  The permutation $\pi^{(k)}$ and its reverse are the shortest permutations with
  minimum gap size equal to $k$.  
  \end{thm}

  \begin{cor}
  Given any $k\in \mathbb{Z}^+$, the permutation $\pi^{(k)} \in S_{(k-1)^2 -2}$ has
  the property that $M_j(p) = \binom{n}{j}$ for all $0\leq j \leq k-2$.
  Furthermore, no permutation of length less than $(k-1)^2 - 2$ has this property.
  \begin{proof}
  Immediate from the construction above, Theorem \ref{otherthm}, and Corollary
  \ref{easy}.  
  \end{proof}
  \end{cor}

  We end this section with one last theorem, a generalization of Theorem \ref{bigthm}.

  \begin{thm} \label{final} 
  Let $p$ be an $n$-permutation with $mg(p) = k+1$, and let $w_k$ be the
  number of pairs $(i,j) \in [n] \times [n]$ such that $|span_p(i,j)| = k-1$. Then
  the number of $(n-k)$-patterns in $p$ is $\binom{n}{k} - w_k$.  \begin{proof}
  Let $p \in S_n$ with $mg(p) = k+1$, and let $i,j \in [n]$ be such that
  $d_p(i,j) = k+1$ (that is, $|span_p(i,j)| = k-1$). Then if we let $S =
  span_p(i,j) \cup i$ and $S' = span_p(i,j) \cup j$, we see that $del(p;S) =
  del(p;S')$, and so $|D_k(p)| \leq \binom{n}{k} - w_k$. 

  For equality, we use a modification of the argument used in Theorem
  \ref{otherthm}. Suppose that $del(p;A) = del(p;B)$ for some $A = \{a_1, a_2,
  \ldots a_k\} \not = B = \{b_1, b_2, \ldots b_k\}$, with $a_i < a_j$ and $b_i <
  b_j$ for $i<j$. Suppose that $a_1 \not = b_1$, and let $s\in[n]$ be the smallest
  integer so that $s \notin A$. Then as in the proof of Theorem \ref{otherthm}, we
  must have that $d_p(a_1,s) = k+1$ and $A-a_1 = B-b_1 = span_p(a_1,s)$. 

  In the case where $a_1 = b_1$, let $p' = del(p,a_1)$, and $A'=A-a_1$, $B' =
  B-b_1$. By Corollary \ref{ratchet}, $mg(p') =k$, since if $mg(p') =
  k+1$, $del(p',A') = del(p',B')$ would contradict Theorem \ref{otherthm}. We now
  repeat the argument, and find that either $a_2 = b_2$ or $A'-a_2 = B'-b_2$. We
  repeat as necessary (no more than $k$ times) to conclude that $|A\cap B| = k-1$. 

  Finally, let $i,j$ be such that $a_i \notin B$ and $b_j \notin A$. It follows
  that $A-a_i = B-b_j = span_p(i,j)$, and so $d_p(i,j) = k+1$. Thus, for each pair
  of entries with distance $k+1$, there are exactly two sets $A,B$ for which
  $del(p;A) = del(p;B)$, and so $|D_k(p)| = \binom{n}{k} - w_k$.  
  \end{proof}
  \end{thm}

\section{Further Questions}

  Considering Wilf's pattern packing problem, we would hope that maximizing the
  large patterns would also maximize the total patterns. For example, having the
  maximum number of patterns of large sizes seems to maximize the total number of
  patterns, but there is some subtlety involved. For example, the two permutations
  in Figure \ref{slant} have been verified to have the maximum number of patterns
  for their size, but not every permutation with the maximum number of patterns
  has the maximum minimum gap. 

  For a concrete example, we see that $|\mathcal{D}(3614725)| =
  |\mathcal{D}(5274136)| = 55$, the maximum number of patterns for permutations
  of size $7$. However, we see that $mg(3614725) = 4$ while $mg(5274136)
  = 3$. Relaxing the requirement that permutations have the maximum number of
  patterns for their size allows us to take this a step further. Setting $p =
  31462758$ and $q = 36147825$, we find that $|\mathcal{D}(p)| = 75$ while
  $|\mathcal{D}(q)| = 89$, though $mg(p) = 3$ and $mg(q)=2$. 

  The data suggests that while maximizing the number of fixed
  size patterns requires a large minimum gap size, the total number of patterns is
  more dependent on the average value of the gaps between pairs of entries. 
  The slanted square construction of \cite{albert} yields a
  permutation with maximum the average gap size between entries, while the
  construction presented here maximizes the minimum gap size.

  Another question which arises is whether or not we can construct a well-behaved
  multivariate generating function which grants us insight into the distribution
  for the number of distinct $(n-k)$-patterns of random permutations as we did
  with the $k=1$ case. However, even if we had the distribution of the minimum gap
  size of random permutations, there is no guarantee that this would translate to
  exact formulas for the distribution of the number of patterns of each length.

\nocite{smith}
\bibliographystyle{plain}
\bibliography{fixedpatterns}

\end{document}